\documentclass[12pt]{amsart}

\usepackage{amsthm, amssymb, amstext, amscd, amsfonts, amsxtra, latexsym, amsmath, comment, 
 mathrsfs, esint, stmaryrd, mathtools}
\usepackage{color}
\usepackage[svgnames]{xcolor}
\usepackage{fancybox}
\usepackage[margin=1.2in]{geometry}
\usepackage[english]{babel}
\usepackage[latin1]{inputenc}
\usepackage{relsize}

\numberwithin{equation}{section}
\newtheorem{theorem}{Theorem}
\newtheorem*{Theorem}{Theorem}
\numberwithin{theorem}{section}
\newtheorem{lemma}[theorem]{Lemma}

\newtheorem{corollary}[theorem]{Corollary}

\theoremstyle{remark}
\newtheorem*{remark}{Remark}

\theoremstyle{definition}

\newcommand{\N}{\mathbb{N}}

\newcommand{\R}{\mathbb{R}}
\newcommand{\Z}{\mathbb{Z}}

\newcommand{\C}{\mathbb{C}}
\newcommand{\OO}{\mathscr{O}}

\newcommand{\e}{\varepsilon}

\newcommand{\li}{{\rm Li}_2}
\renewcommand{\(}{\left(}
\renewcommand{\)}{\right)}
\newcommand{\im}{\operatorname{Im}}
\newcommand{\Log}{\operatorname{Log}}
\newcommand{\OEO}{\overline{OE}}
\newcommand{\OOO}{\overline{\mathscr{O}}}

\allowdisplaybreaks

\renewenvironment{proof}[1][Proof]{\begin{trivlist} \item[\hskip \labelsep {\bfseries #1:}]}{\qed\end{trivlist}}

\author{Min-Joo Jang}
\address{Mathematical Institute\\University of Cologne\\ Weyertal 86-90 \\ 50931 Cologne \\Germany}
\email{min-joo.jang@uni-koeln.de}

\begin{document}

\title{Asymptotic behavior of odd-even partitions}

\begin{abstract}
Andrews studied a function which appears in Ramanujan's identities. 
In Ramanujan's ``Lost'' Notebook, there are several formulas involving this function, but they are not as simple as the identities with other similar shape of functions. 
Nonetheless,  Andrews found out that this function possesses combinatorial information, odd-even partition. In this paper, we provide the asymptotic formula for this combinatorial object.  We also study its companion odd-even overpartitions.
\end{abstract}

\maketitle

\section{Introduction and Statement of results}

Andrews \cite{Andrews} considered a certain family of functions and noticed a mysterious phenomenon. More precisely, Andrews looked into $q$-series identities involving hypergeometric functions, for example in particular (\cite{Andrews} and \cite[Page 19 and Page 104]{Andrews2})
\begin{align}
1+\sum_{n=1}^\infty \frac{q^n}{\left(1-q\right)\left(1-q^2\right)\cdots\left(1-q^n\right)}&=\prod_{n=1}^\infty \frac{1}{\left(1-q^n\right)}, \notag\\
1+\sum_{n=1}^\infty \frac{q^{\frac{n(n+1)}{2}}}{\left(1-q\right)\left(1-q^2\right)\cdots\left(1-q^n\right)}&=\prod_{n=1}^\infty \left(1+q^n\right),  \notag\\
1+\sum_{n=1}^\infty \frac{q^{n^2}}{\left(1-q\right)\left(1-q^2\right)\cdots\left(1-q^n\right)}&=\prod_{n=1}^\infty \frac{1}{\left(1-q^{5n+1}\right)\left(1-q^{5n+4}\right)}, \notag\\
1+\sum_{n=1}^\infty \frac{q^{n}}{\left(1-q^2\right)\left(1-q^4\right)\cdots\left(1-q^{2n}\right)}&=\prod_{n=1}^\infty \frac{1}{\left(1-q^{2n+1}\right)}, \notag\\ 
1+\sum_{n=1}^\infty \frac{q^{n^2}}{\left(1-q^2\right)\left(1-q^4\right)\cdots\left(1-q^{2n}\right)}&=\prod_{n=1}^\infty \left(1+q^{2n+1}\right), \notag \\ 
1+\sum_{n=1}^\infty \frac{q^{\frac{n(n+1)}{2}}}{\left(1-q^2\right)\left(1-q^4\right)\cdots\left(1-q^{2n}\right)}&=?. \label{andrewsft}
\end{align} 
While the others can be nicely written in terms of infinite product (so that it turns out that they are modular forms up to $q$ powers), Andrews did not find any such shape of identities for \eqref{andrewsft}. Moreover, Zagier \cite[Table 1]{Don} figured out that \eqref{andrewsft} is not modular. Nonetheless, Andrews \cite{Andrews} provided a combinatorial interpretation for this function, namely  {\it odd-even partitions}.

 Recall that a {\it partition} of positive integer $n$ is a nonincreasing positive integer sequence whose sum is $n$. Define a partition function $OE(n)$ by the number of partitions of $n$ in which the parts alternate in parity starting with the smallest part odd. In other words, $OE(n)$ counts the number of {\it odd-even partitions of $n$}. For instance, there are no odd-even partitions of $2$ and the odd-even partitions of $3$ are 3 and 2+1, and thus $OE(2)=0$ and $OE(3)=2$. 
Then the generating function for the odd-even partitions is given
\begin{equation}\label{OEgf}
\OO(q):=1+\sum_{n=1}^\infty OE(n)q^n=\sum_{m=0}^\infty\frac{q^{\frac{m(m+1)}{2}}}{\(q^2;q^2\)_m}
\end{equation}
which is exactly identical to \eqref{andrewsft}.
Here the $q$-Pochhammer symbol is defined as $(a)_n:=(a;q)_n:=\prod_{j=1}^n (1-aq^{j-1})$ for $n\in\mathbb{N}_0\cup\{\infty\}$.

In this paper, we investigate the asymptotic behavior of $OE(n)$. 
In order to study the asymptotic behavior of the coefficients of a series, one can either use the Circle Method \cite{circle1, circle3, Wright} or apply Ingham's Tauberian Theorem \cite{Ingham}. Since $\OO(q)$ has a pole at every root of unity and it is not easy to find the bounds for $\OO(q)$ at every root of unity, it is difficult to use the Circle Method in our case. Moreover, as $OE(n)$ is not monotonically increasing, we cannot directly apply Ingham's Tauberian Theorem to our case either (see Section 2 for more details). Thus, we need to slightly modify our function so that we can apply Ingham's Tauberian Theorem.

\begin{theorem}\label{main}
We have
\[
OE(n)\sim  \frac{1}{2\sqrt{5} n^{\frac34}} e^{\pi\sqrt{\frac{n}{5}}}
\]
as $n\to \infty$.
\end{theorem}

We also investigate the asymptotics of {\it odd-even overpartitions}, studied by Lovejoy \cite{Jeremy}. Recall that an {\it overpartition} of positive interger $n$ is a partition of $n$ in which the first occurrence (equivalently, the final occurrence) of a number may be overlined. An {\it odd-even overpartition} is an overpartition with the smallest part odd and such that the difference between successive parts is odd if the smaller is nonoverlined and even otherwise.
For example, there are no odd-even overpartitions of 2, the odd-even overpartitions of $3$ are
$\overline{3},\ 3,\ \overline{2}+1$, and $2+1$, and the odd-even partitions of $4$ are $\overline{3}+\overline{1}$ and $3+\overline{1}$. Notice that if all parts are non-overlined, then we have the odd-even partitions. We denote $\OEO(n)$ by the number of odd-even overpartitions of $n$ and define $\OEO(0):=1$. The generating function is given in \cite{Jeremy} 
\[
\OOO(q):=\sum_{n=0}^\infty {\OEO(n)}q^n= \sum_{m=0}^\infty\frac{(-1)_m q^{\frac{m(m+1)}{2}}}{\(q^2;q^2\)_m}=(-q)_\infty f(q),
\] 
where 
\begin{equation}\label{f(q)}
f(q):=\sum_{n=0}^\infty \frac{q^{n^2}}{(-q)^2_n}
\end{equation}
is one of Ramanujan's third order mock theta functions. These functions appeared in Ramanujan's deathbed letter to Hardy and are now known as the holomorphic parts of weight $1/2$ {\it harmonic Maass forms} (see \cite{Sander}).
We remark that the generating function for the odd-even overpartitions is a {\it mixed mock modular form}, i.e., the product of a modular form and a mock theta function. From this fact, we can apply Wright's Circle Method \cite{Wright} to obtain the asymptotic formula for $\OEO(n)$.

\begin{theorem}\label{od-over}
We have
\[
\OEO(n)\sim  \frac{1}{3^{\frac54} n^{\frac34}} e^{\pi\sqrt{\frac{n}{3}}}
\]
as $n\to \infty$.

\end{theorem}

The paper is organized as follows. In Section 2, we study some basic properties of odd-even partitions and introduce an auxiliary Theorem which play important roles to prove Theorem \ref{main}. The proof is given in Section 3. We conclude the paper with the proof of Theorem \ref{od-over} in Section 4. 

\section*{Acknowledgement}

These results are part of the author's PhD thesis, written under the direction of Kathrin Bringmann. 
The author thanks her for suggesting this problem and valuable advice, Don Zagier and Byungchan Kim for insightful comments and for providing the numerical results to  the main theorems, and Jeremy Lovejoy for affording the idea to consider odd-even overpartitions  which expanded the scope of this paper. 
The author also thanks Steffen L\"obrich and Michael Woodbury for their support and fruitful conversation regarding this topic.

\section{Preliminaries}
\subsection{Basic properties of odd-even partitions}
 First we look into the first few values of the odd-even partition function $OE(n)$: 
\begin{center}
\begin{tabular}{c|c|c}
$n$ & relevant partitions of $n$ & $OE(n)$\\
\hline
\hline
1 & 1 & 1\\
2 & --- & 0\\
3 & 3,\ \ 1+2   & 2\\
4 & --- & 0\\
5 & 5,\ \ 1+4 & 2\\
6 & 1+2+3 & 1\\
7 & 7,\ \ 1+6,\ \ 3+4 & 3\\
8 & 1+2+5 & 1\\
\vdots &\vdots &\vdots
\end{tabular}
\end{center}
From these values, we  see that $OE(n)$ is not monotonically increasing. Nevertheless, $OE(n)\le OE(n+2)$ holds for every $n$ due to the fact that we can always make an odd-even partition of $n+2$ from the one of $n$ by adding 2 to the largest part. Thus, $OE(n)$ is monotonically increasing for even  (odd resp.) $n$. This suggests that the appropriate approach to understand the asymptotic behavior of $OE(n)$ is to split the power series of $OE(n)$ into two parts, one with even $n$ and the other with odd $n$, as follows: 
\[
\OO(q)=\sum_{n=0}^\infty OE(n)q^n=\sum_{n=0}^\infty OE(2n)q^{2n} + \sum_{n=0}^\infty OE(2n+1)q^{2n+1} =:\OO_e(q)+\OO_o(q).
\]
Here, for convenience we define $OE(0):=1$.
We further split the $q$ hypergeometric series in \eqref{OEgf} accordingly by considering the parity of powers of $q$ for each summand. Since the $q$-Pochhammer symbol $(q^2;q^2)_m$ in the denominator always produces even powers of $q$, the parity of powers of $q$ depends only on $m(m+1)/2$. Note that $m(m+1)/2$ is even iff $m\equiv 0,3\pmod{4}$ and odd iff $m\equiv 1,2\pmod{4}$. Hence,
\[
\OO_e(q) =\sum_{m\ge0\atop{m\equiv 0,3\pmod{4}}}\frac{q^{\frac{m(m+1)}{2}}}{\(q^2;q^2\)_m},\qquad
\OO_o(q) =\sum_{m\ge0\atop{m\equiv 1,2\pmod{4}}}\frac{q^{\frac{m(m+1)}{2}}}{\(q^2;q^2\)_m}.
\]

\subsection{Ingham's Tauberian Theorem}
From the asymptotic behavior of a power series, Ingham's Tauberian Theorem \cite{Ingham} gives an asymptotic formula for its coefficients.  

\begin{Theorem}[Ingham \cite{Ingham}]\label{tau} Let $f(q)=\sum_{n\ge0}a(n) q^n$ be a power series with weakly increasing nonnegative coefficients and radius of convergence equal to $1$. If there are constants $A>0,\ \lambda,\alpha\in\R$ such that
\[
f\(e^{-\e}\) \sim\lambda\e^\alpha e^{\frac{A}{\e}}
\]
as $\e\to 0^+$, then
\[
a(n)\sim \frac{\lambda}{2\sqrt{\pi}}\frac{A^{\frac{\alpha}{2}+\frac14}}{n^{\frac{\alpha}{2}+\frac34}} e^{2\sqrt{An}}
\]
as $n\to \infty$.
\end{Theorem}

\section{Proof of Theorem \ref{main}}
 

\subsection{Asymptotics for the generating functions}
In this section, we estimate the functions $\OO_e(q)$ and $\OO_o(q)$. Throughout the section we set $q=e^{-\e}$.
In order to get the asymptotic formulas for these functions, we exploit the second proof of \cite[Proposition 5]{Don}. The idea of the proof is based on the asymptotics of the individual terms in the series. We first study the asymptotic behavior of the summand and then sum up the asymptotics. We denote the $m$th term in the series \eqref{OEgf} by
\[
f_m=f_m(q):=\frac{q^{\frac{m(m+1)}{2}}}{\left(q^2;q^2\right)_m}.
\]
The sequence $(f_m)_{m\in\N}$ is unimodal, meaning that $f_m$ increases until $f_m$ reaches a maximum value and then decreases. More precisely, for $0<|q|<1$ the ratio 
\begin{equation}\label{ratio}
\frac{f_m}{f_{m-1}}=\frac{q^m}{1-q^{2m}}
\end{equation} 
goes to $\infty$ as $m\to0$, decreases as $m$ grows, and tends to $0$ as $m\to\infty$. To determine when $f_m$ takes the maximum value, we check when the ratio \eqref{ratio} becomes $1$.
This ratio is equal to $1$ exactly for $q^{2m}$ the unique root of the equation $Q^{\frac12}+Q=1$ in the interval $(0,1)$, namely $Q:=\frac{3-\sqrt{5}}{2}$. 
In other words, $f_m$ approaches the maximum value when $q^{2m}$ is close to $Q$ and $m$ near $\Log (Q)/(2\Log (q))$. We further note that 
\[
\frac{\Log (Q)}{2\Log (q)}\to \infty,\quad q^{2m}\to Q\qquad \text{as} \quad q\to 1^{-}.
\]
Thus, the main contribution occurs when the terms are of the form $q^{2m}=Qq^{-2\nu}$ (or $q^{m}=Q^{\frac12}q^{-\nu}$) with $\nu\in\nu_0+\Z$ satisfying $\nu=o(m)$ and $\nu_0$ denotes the fractional part of $\Log (Q)/(2\Log (q))$. In this setting, we evaluate the size of $f_m$. For this, we use the asymptotic expansion from Zagier \cite[Page 53]{Don}. Here the {\it dilogarithm function} $\li(z)$ is defined for $|z|<1$ by 
\[
\operatorname{Li}_{2}(z):=\sum_{n=1}^\infty \frac{z^n}{n^2}.
\]

\begin{lemma}\label{donasymp}
Let $A,B\in\R$ and $A>0$. For the unique root $R\in(0,1)$ of the equation $R+R^A=1$ and $q=e^{-\e}$ with $q^n=Rq^{-\nu}$, $\nu=o(n)$ as $n\to\infty$, we have
\begin{equation*}\label{donasmyp}
 \begin{split}
\Log&\left(\frac{q^{\frac12An^2+Bn}}{(q)_n}\right)\\
&=\left(\frac{\pi^2}{6}-\operatorname{Li}_{2}(R)-\frac12\Log(R)\Log(1-R)\right)\e^{-1}-\frac12\Log\left(\frac{2\pi}{\e}\right)+\Log\left(\frac{R^B}{\sqrt{1-R}}\right)\\
&\quad-\left(\frac{A+R-AR}{2(1-R)}\nu^2-\left(B+\frac{R}{2(1-R)}\right)\nu+\frac{1+R}{24(1-R)}\right)\e+O\(\e^2\),
\end{split}
\end{equation*}
as $\e\to0$.
\end{lemma}

\begin{remark}
In fact, Zagier obtained the asymptotic expansion with arbitrary many main terms. Since we only use the first few main terms in this paper, we do not need to consider the complete expansion.
\end{remark}

We set $q\mapsto q^2,\ A\mapsto 1/2$, and $B\mapsto 1/4$ in Lemma \ref{donasymp}. Thus, $R$ becomes $Q$ and we have, recalling that $Q^{\frac12}+Q=1$ and $Q=\frac{3-\sqrt{5}}{2}$, 
\begin{multline}\label{logf}
\Log\left(\frac{q^{\frac{m(m+1)}{2}}}{\left(q^2;q^2\right)_m}\right)
=\left(\frac{\pi^2}{6}-\operatorname{Li}_{2}(Q)-\left(\frac12\Log(Q)\right)^2\right)\frac{1}{2\e}\\-\frac12\Log\left(\frac{\pi}{\e}\right)-\frac{\sqrt{5}}{2}\left(\nu^2-\nu+\frac16\right)\e+O\(\e^2\).
\end{multline}
Furthermore, we use the special value of the dilogarithm function from \cite[Section I.1]{Don}
\begin{equation}\label{li2Q}
\li(Q)=\frac{\pi^2}{15}-\left(\Log\(\frac{1+\sqrt{5}}{2}\)\right)^2
\end{equation}
and note that
\begin{equation}\label{Qrel}
\left(\frac12\Log\left(Q\right)\right)^2=\left(\Log(1-Q)\right)^2=\left(\Log\left((1-Q)^{-1}\right)\right)^2=\left(\Log\(\frac{1+\sqrt{5}}{2}\)\right)^2.
\end{equation}
Combining \eqref{logf}, \eqref{li2Q}, and \eqref{Qrel} gives
\begin{equation}\label{phi}
\begin{split}
\Log\left(\frac{q^{\frac{m(m+1)}{2}}}{\left(q^2;q^2\right)_m}\right)&=\frac{\pi^2}{20\e}-\frac12\Log\left(\frac{\pi}{\e}\right)
-\frac{\sqrt{5}}{2}\left(\nu^2-\nu+\frac16\right)\e + O\(\e^2\) \\
&=\Log\(\varphi(\nu)\)+O\(\e^2\),
\end{split}
\end{equation}
where 
\[
\varphi(\nu):=\sqrt{\frac{\e}{\pi}}\exp\left[\frac{\pi^2}{20\e}-\frac{\sqrt{5}}{2}\left(\nu^2-\nu+\frac16\right)\e\right].
\]

We additionally define for $j\in\{0,1,2,3\}$
\begin{equation*}\label{sigmaj}
\mathcal{S}_j :=\sum_{m\equiv j\pmod{4}} \frac{q^{\frac{m(m+1)}{2}}}{\left(q^2;q^2\right)_m},
\end{equation*}
so that we can write
\begin{equation}\label{eo}
\OO_e(q) =\mathcal{S}_0  +\mathcal{S}_3  ,\qquad
\OO_o(q) =\mathcal{S}_1  +\mathcal{S}_2  .
\end{equation}

\begin{theorem}\label{main2} We have 
\[
\OO_e\(e^{-\e}\)  \sim  \OO_o\(e^{-\e}\) \sim \frac{1}{\sqrt{2\sqrt{5}}} e^{\frac{\pi^2}{20\e}}
\]
as $\e\to0^+$.
\end{theorem}
\begin{proof}
Using \eqref{phi}, we can also rewrite $\mathcal{S}_j$ in terms of $\varphi(\nu)$ as
\begin{equation}\label{sj}
\mathcal{S}_j=\(1+O\(\e^2\)\)\sum_{\nu\equiv \nu_0+j\pmod{4}} \varphi(\nu). 
\end{equation}
To estimate $\mathcal{S}_j$, we begin by rewriting the sum in $\nu$ on the right-hand side of \eqref{sj} as
\begin{equation}\label{theta}
\begin{split}
\sum_{n\in\Z} \varphi(4n+\nu_0+j) &= \sum_{n\in\frac12+\Z} \varphi(4n+\alpha) \\
&=\sqrt{\frac{\e}{\pi}} e^{\frac{\pi^2}{20\e}}\sum_{n\in\frac12+\Z} e^{-\frac{\sqrt{5}}{2}\left(\(4n+\alpha\)^2-\left(4n+\alpha\right)+\frac16\right)\e} \\
&=\sqrt{\frac{\e}{\pi}} e^{\frac{\pi^2}{20\e}-\frac{\sqrt{5}}{2}\(\alpha^2-\alpha+\frac16\)\e} \vartheta\(\frac{\sqrt{5}\(2\alpha-1\)\e i}{\pi}-\frac12;\frac{8\sqrt{5}\e i}{\pi} \),
\end{split}
\end{equation}
where $\alpha:=2+\nu_0+j$ and the Jacobi Theta function is given for $z\in\C$ and $\tau\in\mathbb{H}$ by
\[
\vartheta\(z;\tau\):=\sum_{n\in\frac12+\Z} e^{\pi i n^2\tau+2\pi in \(z+\frac12\)}. 
\]
The modular inversion formula for the Jacobi theta function \cite[Proposition 1.3 (7)]{Sander}
implies that for $a, b\in \C$ with $\operatorname{Re}(a)>0$
\begin{equation*} 
\begin{split}
\vartheta\(\frac{b\e i}{\pi}-\frac12;\frac{a\e i}{\pi}\) &= i\sqrt{\frac{\pi}{a\e}}e^{-\frac{\pi^2}{a\e}\(\frac{b\e i}{\pi}-\frac12\)^2} \vartheta\(\frac{b}{a}+\frac{\pi i}{2a\e};\frac{\pi i}{a\e} \)\\
&=\sqrt{\frac{\pi}{a\e}} \sum_{n\in \Z}(-1)^n e^{-\frac{\pi^2}{a\e} \(n-\frac{b\e i}{\pi}\)^2}.
\end{split}
\end{equation*}
Plugging in $a\mapsto 8\sqrt{5}$ and $b\mapsto \sqrt{5}(2\alpha-1)$ and simplifying the summation  yields that
\begin{equation}\label{thetacal}
\begin{split}
\vartheta\(\frac{\sqrt{5}\(2\alpha-1\)\e i}{\pi}-\frac12;\frac{8\sqrt{5}\e i}{\pi}\)
&=\sqrt{\frac{\pi}{8\sqrt{5}\e}} \sum_{n\in \Z} (-1)^n e^{-\frac{\pi^2}{8\sqrt{5}\e} \(n-\frac{\sqrt{5}(2\alpha-1)\e i}{\pi}\)^2}\\
&=\sqrt{\frac{\pi}{8\sqrt{5}\e}} e^{\frac{\sqrt{5}(2\alpha-1)^2\e}{32}} \(1+O\(\sum_{n\in \Z\setminus\{0\}}e^{-\frac{\pi^2n^2}{8\sqrt{5}\e}}\)\) \\
&=\sqrt{\frac{\pi}{8\sqrt{5}\e}} \(1+O\(\e\)\).
\end{split}
\end{equation}
The last equality comes directly from the fact that as $\e\to 0^+$
\[
e^{\frac{\sqrt{5}(2\alpha-1)^2\e}{32}}=1+ O\(\e \),
\] 
and 
\[
\sum_{n\in \Z\setminus\{0\}}e^{-\frac{\pi^2n^2}{8\sqrt{5}\e}}\ll e^{-\frac{\pi^2}{8\sqrt{5}\e}}.
\]
From \eqref{sj}, \eqref{theta}, and \eqref{thetacal}, we obtain for any $j\in\{0,1,2,3\}$ 
\[
\mathcal{S}_j \sim \frac{1}{2\sqrt{2\sqrt{5}}} e^{\frac{\pi^2}{20\e}-\frac{\sqrt{5}}{2}\(\alpha^2-\alpha+\frac16\)\e}  \sim \frac{1}{2\sqrt{2\sqrt{5}}} e^{\frac{\pi^2}{20\e}}
\]
as $\e\to0^+$.
Recalling \eqref{eo}, we have the desired result.

\end{proof}

Moreover, since $\OO(q)=\OO_e(q)+\OO_o(q)$, we have following Corollary.
\begin{corollary}
We have 
\[
\OO\(e^{-\e}\) \sim    \sqrt{\frac{2}{\sqrt{5}}} e^{\frac{\pi^2}{20\e}}
\]
as  $\e\to0^+$.
\end{corollary}

\begin{remark}
One can directly estimate the series $\OO(q)$ by using the Constant Term Method, inserting
an additional variable to identify the series as the constant term of
the product of more familiar number-theoretic functions in a new variable. (See \cite[First proof of Proposition 5]{Don} for more details.)

\end{remark}

\subsection{Applying Ingham's Tauberian Theorem}

Now we are ready to apply Ingham's Tauberian Theorem to the functions $\OO_e\(e^{-\e}\)$ and $ \OO_o\(e^{-\e}\)$. We first deal with the even case.
Setting $a(n)=OE(2n)$ and replacing $q$ by $q^2$ in Theorem \ref{tau} determines the constants
\[
\lambda=\frac{1}{\sqrt{2\sqrt{5}}},\qquad\alpha=0,\qquad A=\frac{\pi^2}{10}.
\]
We reamrk that since $OE(n)$ does not satisfy weakly increasing property with $n=0$, we only consider when $n\ge 1$.
Thus, we have 
\[
OE(2n) \sim \frac{1}{2\sqrt{5}(2n)^{\frac34}} e^{2\pi\sqrt{\frac{n}{10}}}.
\]
By letting $n\mapsto n/2$, we obtain the desired asymptotic formula for $OE(n)$ with even $n$, namely
\begin{equation}\label{OEeven}
OE(n)\sim  \frac{1}{2\sqrt{5} n^{\frac34}} e^{\pi\sqrt{\frac{n}{5}}}.
\end{equation}

For odd $n$, we rewrite the series as 
\[
\OO_o(q)=\sum_{n=0}^\infty OE(2n+1) q^{2n+1}=q\sum_{n=0}^\infty OE(2n+1)q^{2n}.
\]
Since by Theorem \ref{main2}
\[
\OO_o\(e^{-\e}\) = e^{-\e}\sum_{n=0}^\infty OE(2n+1)e^{-2\e n} \sim  \frac{1}{\sqrt{2\sqrt{5}}} e^{\frac{\pi^2}{20\e}},
\]
we have 
\[
\sum_{n=0}^\infty OE(2n+1)e^{-2\e n} \sim \frac{1}{\sqrt{2\sqrt{5}}} e^{\frac{\pi^2}{20\e}}.
\]
Similar to the case of even $n$, setting $a(n)=OE(2n+1)$ and replacing $q$ by $q^2$ yields
\[
OE(2n+1) \sim \frac{1}{2\sqrt{5}(2n)^{\frac34}} e^{2\pi\sqrt{\frac{n}{10}}}.
\]
As before we let $n\mapsto n/2$ and thus we have for even $n$
\begin{equation}\label{OEodd}
OE(n+1)\sim  \frac{1}{2\sqrt{5} n^{\frac34}} e^{\pi\sqrt{\frac{n}{5}}}.
\end{equation}

Finally from \eqref{OEeven} and \eqref{OEodd} we get the desired asymptotic formula for $OE(n)$, for every $n$,
\[
OE(n)\sim  \frac{1}{2\sqrt{5} n^{\frac34}} e^{\pi\sqrt{\frac{n}{5}}}
\]
as $n\to \infty$.

\section{Proof of Theorem \ref{od-over}}
 We follow the same method of the proof of Theorem 4 in \cite{KB}. The strategy is to estimate the generating function near and away from a dominant pole, and then apply Wright's Circle Method.
Although the method of proof is not new, because we are dealing with a different function, the result does not follow directly from the statement of Theorem 4 in \cite{KB}, and thus we include its proof here. However, it is basically the same proof.

\subsection{Asymptotics of $\OOO(q)$} 
Using the Watson's identity for Ramanujan's third order mock theta function $f(q)$ \cite{Watson}
\[
f(q)=\frac{2}{(q)_\infty} \sum_{n\in\Z} \frac{(-1)^n q^{\frac{n(3n+1)}{2}}}{1+q^n},
\]
we rewrite $\OOO(q)$ as
\begin{equation}\label{ovgf2}
\OOO(q)=\frac{2(-q)_\infty}{(q)_\infty} \sum_{n\in\Z} \frac{(-1)^n q^{\frac{n(3n+1)}{2}}}{1+q^n}.
\end{equation}
From this expression we can see that $\OOO(q)$ has a dominant pole at $q=1$. 

\begin{theorem}\label{ovgfes}
Let $M>0$ fixed.
\begin{enumerate} 
\item [\rm{(i)}] For $|x|\le My$, as $y\to 0^+$
\[
\OOO(q)=\frac{2\sqrt{2}}{3} e^{\frac{\pi i}{24\tau}} +O\(ye^{\frac{\pi}{24}\im\(\frac{-1}{\tau}\)} \).
\]
\item [\rm{(ii)}] For $My<|x|\le 1/2$, as $y\to 0^+$
\[
\OOO(q)\ll\frac{1}{y\sqrt{2}} \exp\left[\frac{1}{y}\(\frac{\pi}{8}-\frac{1}{\pi}\(1-\frac{1}{\sqrt{1+M^2}}\) \)\right].
\]
\end{enumerate}
\end{theorem}

\begin{remark}
One can find $M>\sqrt{\(\frac{12}{12-\pi^2}\)^2-1}=5.543\dots$, so that the bound in the part \rm{(ii)} is indeed an error term.  
\end{remark}

\begin{proof}
\rm{(i)} To estimate the function $\OOO(q)$ near $q=1$, we first examine $f(q)$. By Taylor's theorem, we have
\[
f(q)=f(1)+O\(|\tau|\),
\] 
and from \eqref{f(q)} we see that
\[
f(1)=\sum_{n=0}^\infty \frac{1}{4^n} =\frac 43.
\]
Thus, we have for $|x|\le My$ 
\begin{equation}\label{fcal}
f(q)=\frac43+O\(y\).
\end{equation}
as $y\to 0^+$. 

Now we turn to the infinite product $(-q)_\infty$ in front of $f(q)$.
Recall that, from the modular inversion formula for Dedekind's eta-function (\cite[P.121, Proposition 14]{Kob}), 
\begin{equation}\label{qasym}
(q;q)_\infty=\frac{1}{\sqrt{-i\tau}}e^{-\frac{\pi i\tau}{12}-\frac{\pi i}{12\tau}}\(1+O\(e^{-\frac{2\pi i}{\tau}}\)\).
\end{equation}
Therefore, we find that
\begin{equation}\label{prodest1}
(-q)_\infty=\frac{(q^2;q^2)_\infty}{(q)_\infty}=\frac{1}{\sqrt{2}}e^{\frac{\pi i}{24\tau}} +O\(ye^{\frac{\pi}{24} \text{Im} \( \frac{-1}{\tau} \) }  \).
\end{equation}
Combining \eqref{fcal} and \eqref{prodest1} gives the proof of the part \rm{(i)}.

\rm{(ii)} In the case of $\OOO(q)$  away from $q=1$, we consider the expression in \eqref{ovgf2}. Note that 
\[
\sum_{n\in\Z} \frac{(-1)^n q^{\frac{n(3n+1)}{2}}}{1+q^n}=\frac12+2 \sum_{n\ge 1} \frac{(-1)^n q^{\frac{n(3n+1)}{2}}}{1+q^n}
\]
and that, for $My<|x|\le 1/2$,
\[
\left|\sum_{n\ge 1} \frac{(-1)^n q^{\frac{n(3n+1)}{2}}}{1+q^n} \right|\le \frac{1}{1-|q|} \sum_{n\ge 1}|q|^{\frac{n(3n+1)}{2}} \ll \frac{1}{y}\cdot y^{-\frac12} =y^{-\frac32}.
\]
This implies 
\begin{equation}\label{sumest}
\left|\sum_{n\in\Z} \frac{(-1)^n q^{\frac{n(3n+1)}{2}}}{1+q^n} \right| \ll y^{-\frac32}.
\end{equation}

Now it remains to bound the infinity product
\[
\frac{(-q)_\infty}{(q)_\infty} =\frac{\(q^2;q^2\)_\infty}{(q)^2_\infty}.
\]
We write this as
\begin{align*}
\Log\(\frac{\(q^2;q^2\)_\infty}{(q)^2_\infty}\)
&=\sum_{n\ge1}\(\Log\(1-q^{2n}\)-2\Log\(1-q^n\)\)
=\sum_{n\ge1}\sum_{m\ge1} \frac{2q^{nm}}{m}-\sum_{n\ge1}\sum_{m\ge1} \frac{q^{2nm}}{m}\\
&=\sum_{m\ge1} \(\frac{2q^m}{m\(1-q^m\)}-\frac{q^{2m}}{m\(1-q^{2m}\)} \)
=\sum_{m\ge1} \frac{2q^{2m-1}}{(2m-1)\(1-q^{2m-1}\)}.
\end{align*}
Thus,
\begin{align*}
\left|\Log\(\frac{\(q^2;q^2\)_\infty}{(q)^2_\infty}\)\right|
&\le\sum_{m\ge1} \frac{2|q|^{2m-1}}{(2m-1)\left|1-q^{2m-1}\right|}\\
&\le \sum_{m\ge1} \frac{2|q|^{2m-1}}{(2m-1)\left(1-|q|^{2m-1}\right)}+\frac{2|q|}{|1-q|}-\frac{2|q|}{1-|q|}\\
&=\Log\(\frac{\(|q|^2;|q|^2\)_\infty}{(|q|)^2_\infty}\)-2|q|\(\frac{1}{1-|q|}-\frac{1}{|1-q|}\).
\end{align*}
From \eqref{qasym}, we have 
\[
\frac{\(|q|^2;|q|^2\)_\infty}{(|q|)^2_\infty}= \sqrt{\frac{y}{2}} e^{\frac{\pi}{8y}} \(1+O\(e^{-\frac{\pi}{y}}\)\).
\]
To evaluate the remaining term, we note that for $My< |x| \le \frac12$, $\cos(\pi My)>\cos(\pi x)$.  Therefore,
\[
\left|1-q\right|^2=1-2e^{-2\pi y}\cos(2\pi x)+e^{-4\pi y}> 1-2e^{-2\pi y}\cos(2\pi My)+e^{-4\pi y}.
\]
By the Taylor expansion around $y=0$, we conclude that 
\[
\left|1-q\right|> 2\pi y\sqrt{1+M^2} +O\(y^2\).
\]
Since $1-|q|=2\pi y+O\(y^2\)$, we arrive at
\begin{equation}\label{prodaway}
\left|\frac{\(q^2;q^2\)_\infty}{(q)^2_\infty}\right|\ll \sqrt{\frac{y}{2}}\exp\left[\frac{1}{y}\(\frac{\pi}{8}-\frac{1}{\pi}\(1-\frac{1}{\sqrt{1+M^2}}\)\) \right].
\end{equation}
Plugging \eqref{sumest} and \eqref{prodaway} into \eqref{ovgf2} yields the part \rm{(ii)}.
\end{proof}

\begin{corollary}\label{coro}
For $My<|x|\le 1/2$ with $M>\sqrt{\(\frac{12}{12-\pi^2}\)^2-1}$, these exists $\epsilon>0$ such that as $y\to 0^+$
\[
\OOO(q)\ll\frac{1}{y\sqrt{2}} e^{\frac{\pi}{24}\(\im\(\frac{-1}{\tau}\)-\epsilon\)}.
\]
\end{corollary}

\subsection{Wright's Circle Method}

In this section, we complete the proof of Theorem \ref{od-over} by applying Wright's Circle Method. By Cauchy's Theorem, we see for $y=\frac{1}{4\sqrt{3n}}$ that
\begin{align*}
\OEO&(n)=\frac{1}{2\pi i}\int_\mathcal{C}\frac{\OOO(q)}{q^{n+1}}dq=\int_{-\frac12}^{\frac12} \OOO\(e^{2\pi ix-\frac{\pi}{2\sqrt{3n}}}\)e^{-2\pi inx+\frac{\pi\sqrt{n}}{2\sqrt{3}}} dx\\
&=\int_{|x|\le My}\OOO \(e^{2\pi ix-\frac{\pi}{2\sqrt{3n}}}\)e^{-2\pi inx+\frac{\pi\sqrt{n}}{2\sqrt{3}}}  dx
+\int_{My<|x|\le\frac12}\OOO\(e^{2\pi ix-\frac{\pi}{2\sqrt{3n}}}\)e^{-2\pi inx+\frac{\pi\sqrt{n}}{2\sqrt{3}}}dx\\
&=:\mathcal{I}_1+\mathcal{I}_2,
\end{align*}
where $\mathcal{C}=\{|q|=e^{-\frac{\pi}{2\sqrt{3n}}}\}$. In fact, the integral $\mathcal{I}_1$ contributes the main term as the integral $\mathcal{I}_2$ is an error term.

In order to evaluate $\mathcal{I}_1$, we introduce a function $P_s(u)$, defined by Wright \cite{Wright}, for fixed $M>0$ and $u\in\mathbb{R}^+$
$$
P_{s}(u):=\frac{1}{2\pi i} \int_{1-Mi}^{1+Mi} v^s e^{u\(v+\frac{1}{v}\)} dv.
$$
This functions is rewritten in terms of the $I$-Bessel function up to an error term. 
\begin{lemma}[\cite{Wright}] As $n\rightarrow \infty$
$$
P_s(u)=I_{-s-1}(2u)+O\(e^u\),
$$
where $I_\ell$ denotes the usual the $I$-Bessel function of order $\ell$.
\end{lemma}

Using Theorem \ref{ovgfes} (i), we write the integral $\mathcal{I}_1$ as 
$$
\mathcal{I}_1=\int_{|x|\le \frac{M}{4\sqrt{3n}}}  \(\frac{2\sqrt{2}}{3}e^{\frac{\pi i}{24\tau}}+O\(n^{-\frac12}e^{\frac{\pi\sqrt{n}}{2\sqrt{3}}}\) \)  e^{-2\pi inx+\frac{\pi\sqrt{n}}{2\sqrt{3}}}  dx.
$$
By making the change of variables $v= 1- i 4\sqrt{3n}x$, we arrive at
\begin{align}
\mathcal{I}_1&=\int_{1-Mi}^{1+Mi}   \frac{-i}{4\sqrt{3n}}\(\frac{2\sqrt{2}}{3} e^{\frac{\pi\sqrt{n}}{2\sqrt{3}v}} +O\(n^{-\frac12}e^{\frac{\pi\sqrt{n}}{2\sqrt{3}}} \)\)e^{\frac{\pi \sqrt{n} v}{2\sqrt{3}}} dv\notag\\
&=\frac{\pi\sqrt{2}}{3\sqrt{3n}}P_{0}\(\frac{\pi\sqrt{n}}{2\sqrt{3}}\)+O\(n^{-\frac32}e^{\frac{\pi\sqrt{n}}{\sqrt{3}}}\)
=\frac{\pi\sqrt{2}}{3\sqrt{3n}}I_{-1}\(\frac{\pi\sqrt{n}}{\sqrt{3}}\)+O\(n^{-\frac32}e^{\frac{\pi\sqrt{n}}{\sqrt{3}}}\)\notag\\
&= \frac{1}{3^{\frac54} n^{\frac34}} e^{\frac{\pi\sqrt{n}}{\sqrt{3}}}+O\(n^{-\frac32}e^{\frac{\pi\sqrt{n}}{\sqrt{3}}}\),\label{mainterm}
\end{align}
where we use the asymptotic formula for the $I$-Bessel function \cite[4.12.7]{AAR}
\[
I_\ell(x)=\frac{e^x}{\sqrt{2\pi x}}+O\(\frac{e^x}{x^{\frac32}}\).
\]

Now we turn to the integral $\mathcal{I}_2$. From the Corollarly \ref{coro}, we have for $My<|x|\le 1/2$
\[
\mathcal{I}_2\ll  \int_{My<|x|\le\frac12}2\sqrt{6n} e^{\frac{1}{y}\(\frac{\pi}{24}-\epsilon\)} e^{\frac{\pi\sqrt{n}}{2\sqrt{3}}}dx \ll n^{\frac12} e^{\frac{\pi\sqrt{n}}{\sqrt{3}}(1-\epsilon)},
\]
which together with \eqref{mainterm} completes the proof.

\end{document}